\documentclass [10pt]{article}
\usepackage[english]{babel}
\usepackage[latin1]{inputenc} 
\usepackage{amssymb,amsmath,amsthm}

\newtheorem{theorem}{Theorem}[section]
\newtheorem{corollary}[theorem]{Corollary}

\newtheorem{remark}{Remark}

\usepackage{hyperref}

\title{On nearly K\"{a}hler and  K\"{a}hler-Codazzi type manifolds}
\author{Fernando Etayo\footnote{Departamento de  Matem\'{a}ticas, Estad\'{\i}stica y Computaci\'{o}n.  Universidad de Cantabria. Avda. de los Castros, s/n, 39071 Santander, SPAIN. e-mail: etayof@unican.es}, Araceli deFrancisco\footnote{Departamento de Matem\'{a}ticas. Escuela de Ingenier\'{\i}as Industrial, Inform\'{a}tica y Aeroespacial. Universidad de Le\'{o}n. Campus de Vegazana, 24071 Le\'{o}n, SPAIN. e-mail: afrai@unileon.es} \, and Rafael Santamar\'{\i}a\footnote{Departamento de Matem\'{a}ticas. Escuela de Ingenier\'{\i}as Industrial, Inform\'{a}tica y Aeroespacial. Universidad de Le\'{o}n. Campus de Vegazana, 24071 Le\'{o}n, SPAIN. e-mail: rsans@unileon.es}}
\date{}
\begin{document}
\maketitle

\begin{abstract}
Nearly K\"{a}hler and K\"{a}hler-Codazzi type manifolds are defined in a very similar way.  We prove that nearly K\"{a}hler type manifolds have sense just in Hermitian and para-Hermitian contexts, and that K\"{a}hler-Codazzi type manifolds reduce to K\"{a}hler type manifolds in all the four Hermitian, para-Hermitian, Norden and product Riemannian geometries.
\end{abstract}



\section{Introduction}

Nearly K\"{a}hler manifolds were studied by  many authors (see, for instance, the classical work \cite{gray} of Gray). They form a class between those of almost Hermitian and K\"{a}hler manifolds. The six sphere is a nice example of a nearly K\"{a}hler manifold which is not a K\"{a}hler manifold.  In the almost para-Hermitian case there also exist strict nearly para-K\"{a}hler manifolds. The Libermann's quadric $S^{6}_{3}$ is also a nice example of nearly para-K\"{a}hler non para-K\"{a}hler manifold (see \cite[Ex. 3.7]{gadea}).

In this paper we will prove that the condition defining nearly K\"{a}hler type manifolds when applied in almost Norden or almost product Riemannian manifolds leads to a K\"{a}hler type condition. The same study will be done for the condition defining Codazzi-K\"{a}hler type manifolds in the Norden and product Riemannian cases, proving that it reduces to a K\"{a}hler type condition in all the four geometries, thus there not being strict Codazzi-K\"{a}hler type manifolds. 

The above result is the primary goal of the present paper. Now we  point out other achieved results throughout  the paper and we show its organization at a time. Section \ref{sec:preliminaries} contains the definitions and known results necessary to fulfill the objectives set out. In particular, we will recall the notion of $(J^2=\pm1)$-metric manifold which is a common framework for the four aforementioned geometries. We will also recall the definition of the first canonical connection of a such manifold. In Section \ref{sec:no-characterizations} we will show how the torsion tensor of the first canonical connection of a $(J^2=\pm1)$-metric manifold characterizes K\"{a}hler type and integrable manifolds (see Theorem \ref{thm:toro-characterizations}). In Section \ref{sec:nearly-kahler-manifolds} we will prove that the class of nearly K\"{a}hler type manifolds in the almost Norden and almost product Riemannian cases is the class of the K\"{a}hler type manifolds (see Theorem \ref{thm:nearly-kahler+1}). We will finish this section showing how the torsion tensor of the first canonical connection characterizes the class of nearly K\"{a}hler type manifolds in the almost Hermitian and almost para-Hermitian cases (see Theorem \ref{thm:nearly-kahler-toro-1}).  Section \ref{sec:kahler-codazzi-type-manifolds}  is devoted to study the class of K\"{a}hler-Codazzi type manifolds, previously introduced in Section \ref{sec:preliminaries}, and to achieve the primary goal of this paper (see Theorem \ref{thm:kahler-codazzi}). As a direct consequence of this theorem we will conclude that there are no strict Codazzi-K\"{a}hler type manifolds in the Norden case. Finally, in Section \ref{sec:codazzi-couplings}, we will recall the notion of Codazzi-coupled connection. In \cite{fei-zhang}, the authors introduce the notion of Codazzi-coupled connection on almost Hermitian or almost para-Hermitian manifolds as a connection, not necessary torsion free, that fullfills the Codazzi-coupled conditions \eqref{eq:codazzi-coupled}. We will use one of the main results of the quoted paper to show another demonstration of Theorem \ref{thm:kahler-codazzi} on almost Hermitian and almost para-Hermitian manifolds.

\section{Preliminaries}
\label{sec:preliminaries}

A manifold will be called to have an $(\alpha, \varepsilon)$-structure $(J,g)$ if $J$ is an almost complex ($\alpha =-1$) or almost product ($\alpha =1$) structure and $J$ is an isometry ($\varepsilon =1$) or anti-isometry ($\varepsilon =-1$) of a semi-Riemannian metric $g$. The metric $g$ will be a Riemannian metric if $\varepsilon=1$. It is also said that $(M,J,g)$ is a $(J^{2}=\pm 1)$-metric manifold. Thus, there exist four kinds of $(\alpha ,\varepsilon)$-structures according to the values $\alpha ,\varepsilon \in \{-1,1\}$, where
\begin{equation}
\label{eq:ae-structure}
J^2 = \alpha Id, \enspace g(JX,JY)= \varepsilon g(X,Y), \enspace \forall X, Y \in \mathfrak{X}(M),
\end{equation}
$Id$ denotes the identity tensor field and $\mathfrak{X}(M)$ denotes the set of vector fields on $M$. 

In the above conditions, it is easy to prove the following equivalence
\begin{equation}
\label{eq:alternative-metric-condition}
g(JX,JY)=\varepsilon g(X,Y) \Longleftrightarrow g(JX,Y)=\alpha\varepsilon g(X, JY), \enspace \forall X, Y \in \mathfrak{X}(M).
\end{equation}

The corresponding manifolds are known as:
\begin{enumerate}
\renewcommand*{\theenumi}{\roman{enumi})}
\renewcommand*{\labelenumi}{\theenumi}

\item  Almost-Hermitian manifold if it has a  $(-1,1)$-structure. 

\item Almost product Riemannian manifolds if it has an $(1,1)$-structure. We shall consider through this paper that the trace of $J$ vanishing, which in particular means these manifolds have even dimension. Almost para-Norden manifolds is another denomination for this kind of metric manifolds having even dimension (see, {\em e.g.}, \cite[Defin. 2.1]{ida-manea}). 

\item Almost anti-Hermitian or almost Norden manifolds if it has a $(-1,-1)$-structure. 

\item Almost para-Hermitian manifolds if it has an $(1,-1)$-structure. 
\end{enumerate}
In the last two cases the metric $g$ is semi-Riemannian having signature $(n,n)$.

K\"{a}hler type and integrable manifolds are the most studied classes in each of the four geometries 
aforementioned. Let $(M,J,g)$ be a manifold  endowed with an $(\alpha,\varepsilon)$-structure. Let us denote by $\nabla^{\mathrm{g}}$ the Levi Civita connection of $g$. The manifold $(M,J,g)$ is said to be a K\"{a}hler type manifold if $\nabla^{\mathrm{g}}J=0$. The manifold $(M,J,g)$ is said to be integrable if the Nijenhuis tensor of the tensor field $J$ vanishes. 

It is well-known that the Nijenhuis tensor of $J$, $N_J$, and the Levi Civita connection of $g$ satisfy the next relation
\begin{equation}
\label{eq:nijenhuis-nablaJ}
N_J(X,Y)=(\nabla^{\mathrm{g}}_X J) JY + (\nabla^{\mathrm{g}}_{JX} J) Y - (\nabla^{\mathrm{g}}_Y J)JX- (\nabla^{\mathrm{g}}_{JY} J) X, \enspace \forall X, Y \in \mathfrak{X}(M).
\end{equation}
Then, the Levi Civita also plays an important role in the characterization of integrable manifolds. A $(J^2=\pm1)$-metric manifold $(M,J,g)$ is an integrable manifold if and only if equation \eqref{eq:nijenhuis-nablaJ} vanishes.

Nevertheless, there exist other connections that allow to characterize the above two classes. One of them is the first canonical connection, firstly introduced in the Hermitian geometry  as follows
\begin{equation*}
\nabla^0_X Y = \nabla^{\mathrm{g}}_X Y + \frac{1}{2} (\nabla^{\mathrm{g}}_X J) JY, \enspace X, Y \in \mathfrak{X}(M),
\end{equation*}
(see \cite{gauduchon,lichnerowicz}), and later extended to the other three $(\alpha,\varepsilon)$ geometries as
\begin{equation}
\label{eq:first-canonical}
\nabla^{0}_X Y= \nabla^{\mathrm{g}}_X Y +\frac{(-\alpha)}{2} (\nabla^{\mathrm{g}}_X J)JY,   \enspace \forall X,Y \in \mathfrak{X} (M),
\end{equation}
(see \cite{brno}). 

If $\alpha\varepsilon=-1$, i.e., in the almost Hermitian and the almost para-Hermitian cases, there is another class of manifolds carefully studied as well: nearly K\"{a}hler type manifolds. This class can be characterized with the help of the Levi Civita connection by one of the  two equivalent conditions
\begin{equation}
\label{eq:nearly-Kahler}
(\nabla^{\mathrm{g}}_X J)X=0, \enspace (\nabla^{\mathrm{g}}_X J)Y + (\nabla^{\mathrm{g}}_Y J)X=0, \enspace \forall X, Y \in \mathfrak{X}(M).
\end{equation}
This class of manifolds can be introduced on $(J^2=\pm1)$-metric manifolds in the same way, i.e.,  a $(J^2=\pm1)$-metric manifold  $(M,J,g)$ will be called a nearly K\"{a}hler type manifold if the Levi Civita connection $\nabla^{\mathrm{g}}$ and the tensor field $J$ satisfy the equivalent conditions \eqref{eq:nearly-Kahler}. 

The common technique that allows to classify the classes of manifolds in the four geometries corresponding with the notion of $(J^2=\pm1)$-metric manifold is based in the study of the vectorial subspaces of the next subspace of $V^*\otimes V^*\otimes V^*$
\[
W=\{ \varphi \in V^*\otimes V^*\otimes V^*\colon \varphi(x,y,z)=\alpha\varepsilon \varphi (x,z,y),  \varphi(x,Jy,z)=-\alpha\varepsilon \varphi (x,y,Jz), \forall x, y, z \in V \},
\]
where $(V, J,<,>)$ is $2n$-dimensional real vectorial space, $J$ is an endomorphism of $V$ and $<,>$ is an inner product satisfying the same identities that $J$ and $g$ described in \eqref{eq:ae-structure}  (see \cite{gadea,gray-hervella,ganchev-borisov,staikova}). Note that the elements of the subspace $W$ have the same symmetries that the tensor $g((\nabla^{\mathrm{g}}_X J)Y,Z), X, Y,Z \in \mathfrak{X}(M)$ (see equations \eqref{eq:nablaJ2} and \eqref{eq:nablaJ3}). In the case $\alpha\varepsilon=-1$,  the subspace $W_1$ of $W$ defined as follows 
\[
W_1=\{\varphi \in W \colon \varphi(x,x,y)=0, \forall x, y \in V\},
\]
allows to introduce the class of nearly K\"{a}hler type manifolds. It is easy to prove that the subspace $W_1$ also can be defined in the next way
\[
W_1=\{\varphi \in W \colon \varphi(x,y,z)+\varphi(y,x,z)=0, \forall x, y, z \in V\}.
\]
Both definitions correspond with the equivalent conditions \eqref{eq:nearly-Kahler}. In the case $\alpha\varepsilon=1$, the subspace $W_1$ defined as above is the zero subspace. Indeed, given $\varphi \in W_1$ and $x,y,z \in V$, one has
\[\varphi(x,y,z)= \varphi(x,z,y)=-\varphi(z,x,y), \enspace 
\varphi(x,y,z)=-\varphi(y,x,z)=-\varphi(y,z,x)=\varphi(z,y,x)=\varphi(z,x,y),
\]
then $\varphi(x,y,z)=0$, thus $W_1=\{0\}$. This fact explains why nearly K\"{a}hler type manifolds in the case $\alpha\varepsilon=1$ are K\"{a}hler type manifolds.

In the Mathematical Literature one can find other attempts to introduce in the case $\alpha\varepsilon=1$ an analogous class to the nearly K\"{a}hler one if $\alpha\varepsilon=-1$. Given an almost Hermitian or almost para-Hermitian manifold $(M,J,g)$, one can define the fundamental two form 
\begin{equation*}
\omega(X,Y)=g(JX,Y), \enspace \forall X, Y \in \mathfrak{X}(M).
\end{equation*}
Thus, nearly K\"{a}hler type manifolds can also be characterized with the help of the fundamental two form  by the condition
\begin{equation*}
(\nabla^{\mathrm{g}}_X \omega)(Y,Z)+ (\nabla^{\mathrm{g}}_Y \omega)(X,Z)=0, \enspace \forall X, Y, Z \in \mathfrak{X}(M).
\end{equation*}
Almost anti-Hermitian and almost product Riemannian cases are quite different. Given a manifold $(M,J,g)$ in the previous conditions, the tensor
\begin{equation*}
\widetilde{g}(X,Y)=g(JX,Y), \enspace \forall X, Y \in \mathfrak{X}(M),
\end{equation*}
defines another (semi)-Riemannian metric instead of a two form like in the $\alpha\varepsilon=-1$ case, called the twin metric of $g$. The following equality
\begin{equation}
\label{eq:codazzi-equation}
(\nabla^{\mathrm{g}}_X \widetilde{g})(Y,Z)- (\nabla^{\mathrm{g}}_Y \widetilde{g})(X,Z)=0, \enspace \forall X, Y, Z \in \mathfrak{X}(M),
\end{equation}
is called the Codazzi equation. In the case $\alpha=\varepsilon=-1$, manifolds satisfying the Codazzi equation are called  anti-K\"{a}hler-Codazzi manifolds, while in the case $\alpha=\varepsilon=1$, this kind of manifolds are called para-K\"{a}hler-Norden-Codazzi manifolds. First, in the Norden case, they were introduced  in \cite{ST} and have been intensively studied (see also \cite{S14,S18}). Afterward, the class of manifolds characterized by the Codazzi equation were extended without changes in \cite{ida-manea} to the other $\alpha\varepsilon=1$ case, the product Riemannian case.

It is easy to prove that Codazzi equation \eqref{eq:codazzi-equation} is equivalent to the next one
\begin{equation}
\label{eq:codazzi-equation-2}
(\nabla^{\mathrm{g}}_{X}J)Y-(\nabla^{\mathrm{g}}_{Y}J)X= 0,\enspace \forall X,Y\in \mathfrak{X}(M).
\end{equation}

For the sake of simplicity a $(J^2=\pm1)$-metric manifold such that $\alpha\varepsilon=1$ satisfying   equation \eqref{eq:codazzi-equation-2} will be named a K\"{a}hler-Codazzi type manifold.
 
\section{Characterizations of K\"{a}hler type and integrable manifolds by means of the first canonical connection}
\label{sec:no-characterizations}

The first canonical connection of a $(J^2=\pm1)$-metric manifold $(M,J,g)$ parallelizes both $J$ and $g$, i.e., $\nabla^0 J=0$ and $\nabla^0 g=0$ (see \cite[Lemma 3.10]{brno}), but in general has torsion.  As direct consequence of identity \eqref{eq:first-canonical} one can prove that the torsion tensor $\mathrm{T}^0$ of $\nabla^0$ satisfies the following one
\begin{equation}
\label{eq:torsion0}
\mathrm{T}^{0}(X,Y)=\frac{(-\alpha)}{2} ((\nabla^{\mathrm{g}}_X J)JY-(\nabla^{\mathrm{g}}_Y J) JX),   \enspace\forall X,Y \in \mathfrak{X} (M).
\end{equation}
Therefore, straightforward calculations allow to conclude that the tensors $\mathrm{T}^0$ and $N_J$ are related by the next equality
\begin{equation*}
-\frac{1}{2} N_J(X,Y)=\mathrm{T}^{0}(JX,JY) +\alpha \mathrm{T}^{0} (X,Y), \enspace \forall X, Y \in \mathfrak{X}(M).
\end{equation*} 
The above properties and identities allow to characterize K\"{a}hler type and integrable manifolds as follows.

\begin{theorem}
\label{thm:toro-characterizations}
 Let $(M,J,g)$ be a $(J^2=\pm1)$-metric manifold. 
\begin{enumerate}
\renewcommand*{\theenumi}{\roman{enumi})}
\renewcommand*{\labelenumi}{\theenumi}

\item The manifold $(M,J,g)$ is a K\"{a}hler type manifold if and only if the torsion tensor of the first canonical connection vanishes.

\item The manifold $(M,J,g)$ is an integrable manifold if and only if the next relation holds
\[
\mathrm{T}^{0}(JX,JY) +\alpha \mathrm{T}^{0} (X,Y)=0, \enspace \forall X, Y \in \mathfrak{X}(M).
\]
\end{enumerate}

\end{theorem}

\begin{remark} We finish this section highlighting the next property of the tensor field $\nabla^{\mathrm{g}}J$ on a $(J^2=\pm1)$-metric manifold:
\begin{equation}
\label{eq:nablaJ}
(\nabla^{\mathrm{g}}_{X}J)JY =-J(\nabla^{\mathrm{g}}_{X}J)Y, \enspace \forall X, Y \in \mathfrak{X}(M).
\end{equation}
Indeed, given $X, Y$ vectors fields on $M$ and taking into account that $J^2=\alpha Id$, one has 
\begin{equation*}
(\nabla^{\mathrm{g}}_{X}J)JY =  \alpha \nabla^{\mathrm{g}}_{X} Y - J  \nabla^{\mathrm{g}}_{X} JY = J^2 (\nabla^{\mathrm{g}}_{X} Y) - J  \nabla^{\mathrm{g}}_{X} JY = J \left(J\nabla^{\mathrm{g}}_{X} Y -  \nabla^{\mathrm{g}}_{X} JY\right)=-J(\nabla^{\mathrm{g}}_{X}J)Y.
\end{equation*}

Identities \eqref{eq:torsion0} and \eqref{eq:nablaJ} allow us to write the torsion tensor of $\nabla^0$ as follows
\begin{equation}
\label{eq:torsion0-2}
\mathrm{T}^{0}(X,Y)=\frac{\alpha}{2}J\left((\nabla^{\mathrm{g}}_{X}J)Y-(\nabla^{\mathrm{g}}_{Y}J)X\right), \enspace \forall X, Y \in \mathfrak{X}(M).
\end{equation}
\end{remark}

\section{The class of nearly-K\"{a}hler type manifolds}
\label{sec:nearly-kahler-manifolds}

First we recall the following two properties of manifolds endowed with an $(\alpha,\varepsilon)$ structure.

\begin{remark} Let $(M,J,g)$ be $(J^2=\pm1)$-metric. Given $X,Y,Z$ vector fields on $M$, then taking into account  the equivalence \eqref{eq:alternative-metric-condition} one has the following identities
\begin{align*}
0=(\nabla^{\mathrm{g}}_X g)(JY,Z) &= Xg(JY,Z)-g(\nabla^{\mathrm{g}}_X JY, Z)-g(JY,\nabla^{\mathrm{g}}_X Z)  \\
&= Xg(JY,Z)-g(\nabla^{\mathrm{g}}_X JY, Z)- \alpha\varepsilon g(Y,J \nabla^{\mathrm{g}}_X Z),  \\
0=\alpha\varepsilon (\nabla^{\mathrm{g}}_X g)(Y,JZ) &= \alpha\varepsilon \left( Xg(Y,JZ)-g(\nabla^{\mathrm{g}}_X Y, JZ)-g(Y,\nabla^{\mathrm{g}}_X JZ)\right)  \\
&= Xg(JY,Z)- g(J\nabla^{\mathrm{g}}_X Y, Z)- \alpha\varepsilon g(Y, \nabla^{\mathrm{g}}_X JZ),
\end{align*}
then subtracting the above equalities one obtains
\begin{equation}
\label{eq:nablaJ2}
g((\nabla^{\mathrm{g}}_X J)Y,Z) = \alpha\varepsilon g((\nabla^{\mathrm{g}}_X J)Z,Y).
\end{equation}
Moreover, if one combines the properties \eqref{eq:alternative-metric-condition}  and \eqref{eq:nablaJ} one also obtains the next identity
\begin{equation}
\label{eq:nablaJ3}
g((\nabla^{\mathrm{g}}_X J)JY,Z)=- g(J(\nabla^{\mathrm{g}}_X J)Y,Z) = - \alpha\varepsilon g((\nabla^{\mathrm{g}}_X J)Y,JZ), \enspace \forall X, Y, Z \in \mathfrak{X}(M).
\end{equation}
\end{remark}

As we have recalled in the Preliminaries, nearly K\"{a}hler type manifolds in the almost Hermitian and the almost para-Hermitian cases can be characterized by one of the two equivalent conditions \eqref{eq:nearly-Kahler}. Now, we will study the class of manifolds satisfying one of these two conditions in the case $\alpha\varepsilon=1$.

\begin{theorem}
\label{thm:nearly-kahler+1}
 Let $(M,J,g)$ be a $(J^2=\pm1)$-metric manifold satisfying $\alpha\varepsilon=1$. If the Levi Civita connection $\nabla^{\mathrm{g}}$ and the tensor field $J$ satisfy 
\[
(\nabla^{\mathrm{g}}_X J)Y + (\nabla^{\mathrm{g}}_Y J)X=0, \enspace \forall X, Y \in \mathfrak{X}(M),
\]
then $(M,J,g)$ is a K\"{a}hler type manifold.
\end{theorem}

\begin{proof} Given $X,Y,Z$ vector fields on $M$, taking into account the above condition and property \eqref{eq:nablaJ2} in the case $\alpha\varepsilon=1$ one obtains
\begin{align*}
g((\nabla^{\mathrm{g}}_X J) Y, Z) & = g((\nabla^{\mathrm{g}}_X J) Z, Y)=-g (\nabla^{\mathrm{g}}_Z J) X, Y),\\
g((\nabla^{\mathrm{g}}_X J) Y, Z) & = - g((\nabla^{\mathrm{g}}_Y J) X, Z) = -g((\nabla^{\mathrm{g}}_Y J) Z, X) = g((\nabla^{\mathrm{g}}_Z J) Y, X) = g((\nabla^{\mathrm{g}}_Z J) X, Y),
\end{align*}
then $g((\nabla^{\mathrm{g}}_X J) Y, Z)=0$, and thus, one can conclude that $(M,J,g)$ is a K\"{a}hler type manifold.
\end{proof}

As in the case of K\"{a}hler type and integrable $(J^2=\pm1)$-metric manifolds, nearly-K\"{a}hler type manifolds in the $\alpha\varepsilon=-1$ case can be easily characterized by means of the first canonical connection as follows.

\begin{theorem}
\label{thm:nearly-kahler-toro-1}
 Let $(M,J,g)$ be a $(J^2=\pm1)$-metric manifold satisfying $\alpha\varepsilon=-1$. Then $(M,J,g)$ is a nearly K\"{a}hler type manifold if and only if 
\[
g(\mathrm{T}^0(X,Y),X)=0, \enspace X, Y \in \mathfrak{X}(M).
\]

\end{theorem}

\begin{proof} Given $X, Y$ vector fields on $M$, if $\alpha\varepsilon=-1$ then, taking into account \eqref{eq:nablaJ2} and \eqref{eq:nablaJ3}, one has
\[
g((\nabla^{\mathrm{g}}_{Y})JX, X)= g((\nabla^{\mathrm{g}}_{Y})X, JX) = -g((\nabla^{\mathrm{g}}_{Y})JX, X),
\]
i.e., $g((\nabla^{\mathrm{g}}_{Y})JX, X)=0$. Taking into account \eqref{eq:torsion0}, \eqref{eq:nablaJ2}  and the last identity one obtains the next equality
\[
g(\mathrm{T}^{0}(X,Y),X)=\frac{(-\alpha)}{2} ( g((\nabla^{\mathrm{g}}_{X}J)JY,X)-g((\nabla^{\mathrm{g}}_{Y}J)JX,X))) = \frac{\alpha}{2} g((\nabla^{\mathrm{g}}_{X}J)X,JY), \enspace \forall X, Y \in \mathfrak{X}(M),
\]
thus, one can conclude that the equivalence of this statement is true.
\end{proof}

\section{The class of K\"{a}hler-Codazzi type manifolds in the $\alpha\varepsilon=1$ case}
\label{sec:kahler-codazzi-type-manifolds}

Obviously, in the $\alpha\varepsilon=1$ case, the Levi Civita connection of any K\"{a}hler type manifold satisfies equation \eqref{eq:codazzi-equation-2}, thus any K\"{a}hler type manifold being a K\"{a}hler-Codazzi type manifold. In \cite[Prop. 2.1]{ida-manea}, the authors prove that any para-K\"{a}hler-Norden-Codazzi manifold in the almost product  Riemannian case is a para-K\"{a}hler-Norden manifold, i.e., the Levi-Civita connection satisfies $\nabla^{\mathrm{g}}J=0$. In the almost  anti-Hermitian case there is no analogous result. However, examples of strict anti-K\"{a}hler-Codazzi manifolds in this case, i.e.,  $\nabla^{\mathrm{g}} J\neq0$, are not yet shown.  Now we  prove that every anti-K\"{a}hler-Codazzi manifold is an anti-K\"{a}hler manifold. In general, we prove that every manifold having an $(\alpha,\varepsilon)$-structure such that its Levi Civita connection satisfies identity \eqref{eq:codazzi-equation-2} is a K\"{a}hler type manifold, which is a direct consequence of the last identity of Section \ref{sec:no-characterizations} as we will see below.

\begin{theorem}
\label{thm:kahler-codazzi}
Let $(M,J,g)$ be a $(J^2=\pm1)$-metric manifold. If the Levi Civita connection $\nabla^{\mathrm{g}}$ and the tensor field $J$ satisfy identity \eqref{eq:codazzi-equation-2} then $(M,J,g)$ is a K\"{a}hler type manifold.
\end{theorem}

\begin{proof}

If $\nabla^{\mathrm{g}}$ satisfies $(\nabla^{\mathrm{g}}_{X}J)Y-(\nabla^{\mathrm{g}}_{Y}J)X= 0$, for all vector fields $X, Y$ on $M$, then, taking into account \eqref{eq:torsion0-2}, one concludes that the torsion tensor $\mathrm{T}^{0}$ vanishes, and therefore, $(M,J,g)$ is K\"{a}hler type manifold (see Theorem \ref{thm:toro-characterizations}).
\end{proof}

The above theorem shows that every K\"{a}hler-Codazzi type manifold in the case $\alpha\varepsilon=1$ is a K\"{a}hler type manifold. In particular, we obtain that anti-K\"{a}hler-Codazzi manifolds in the sense of \cite{ST} are K\"{a}hler type manifolds.

\begin{corollary} Let $(M,J,g)$ be an anti-Hermitian manifold. Every anti-K\"{a}hler-Codazzi manifold is a anti-K\"{a}hler manifold, i.e., the Levi Civita connection $\nabla^{\mathrm{g}}$ and the tensor field $J$ satisfy $\nabla^{\mathrm{g}}J=0$. Therefore, there are no strict anti-K\"{a}hler-Codazzi manifolds.
\end{corollary}

\begin{remark} Theorem \ref{thm:kahler-codazzi} allows us to recover the analogous result for the almost product Riemannian case proved in \cite[Prop. 2.1]{ida-manea} and previously recalled: Every para-K\"{a}hler-Norden-Codazzi manifold $(M,J,g)$ is a para-K\"{a}hler-Norden manifold. Ida and Manea use the properties of tensor $g((\nabla^{\mathrm{g}}_X J)Y,Z), X, Y,Z \in \mathfrak{X}(M)$, in the case $\alpha\varepsilon=1$ (see equations \eqref{eq:nablaJ2} and \eqref{eq:nablaJ3}) and identity \eqref{eq:codazzi-equation-2} to demonstrate this in a similar way to the proof of Theorem \ref{thm:nearly-kahler+1} of this paper.
We propose a unified proof of both results in the $\alpha\varepsilon=1$ case taking into account that the equivalent condition to the Codazzi equation \eqref{eq:codazzi-equation}, 
\[
(\nabla^{\mathrm{g}}_{X}J)Y-(\nabla^{\mathrm{g}}_{Y}J)X= 0,\enspace \forall X,Y\in \mathfrak{X}(M),
\]
also allows to show that the torsion of the first canonical connection $\nabla^0$ vanishes,
\[
\mathrm{T}^{0}(X,Y)=\frac{(-\alpha)}{2} ((\nabla^{\mathrm{g}}_X J)JY-(\nabla^{\mathrm{g}}_Y J) JX)=\frac{\alpha}{2}J\left((\nabla^{\mathrm{g}}_{X}J)Y-(\nabla^{\mathrm{g}}_{Y}J)X\right),   \enspace\forall X,Y \in \mathfrak{X} (M).
\]
\end{remark}

\section{On Codazzi couplings on $(J^2=\pm1)$-metric manifolds in the case $\alpha\varepsilon=-1$}
\label{sec:codazzi-couplings}

In a more general setting, in \cite{fei-zhang}, the authors introduce a relaxation of the parallelism conditions over a connection on a certain kind of manifolds. In the particular case of an almost Hermitian or almost para-Hermitian manifold $(M,J,g)$, i.e., in a manifold having an $(\alpha,\varepsilon)$-structure satisfying $\alpha\varepsilon=-1$, they introduce the notion of Codazzi-coupled connection as a connection $\nabla$ on $M$, not necessary torsion-free connection, that satisfies the below conditions, named the Codazzi-coupled conditions,
\begin{equation}
\label{eq:codazzi-coupled} 
(\nabla _{X}J)Y-(\nabla _{Y}J)X= 0, \enspace (\nabla _{Z}g)(X,Y)-(\nabla _{X}g)(Z,Y)=0,\enspace \forall X,Y,Z\in \mathfrak{X}(M). 
\end{equation}
They also prove the following theorem about this kind of manifolds having a Codazzi-coupled torsion-free connection.

\begin{theorem}[{\cite[Theor. 3.2]{fei-zhang}}]
Let $(M,J,g)$ be an almost almost Hermitian or almost para-Hermitian manifold endowed with a torsion-free connection $\nabla$ satisfying
\begin{equation*}
(\nabla _{X}J)Y-(\nabla _{Y}J)X= 0, \enspace (\nabla _{Z}g)(X,Y)-(\nabla _{X}g)(Z,Y)=0,\enspace \forall X,Y,Z\in \mathfrak{X}(M). 
\end{equation*}
Then $(M,J,g)$ is a K\"{a}hler or para-K\"{a}hler manifold.
\end{theorem}

An almost Hermitian or almost para-Hermitian manifold endowed with a Codazzi-coupled torsion-free connection is called an Codazzi-K\"{a}hler or Codazzi-para-K\"{a}hler manifold (see \cite[Defin. 3.8]{fei-zhang}).

The above statement is given just for the two geometries obtained with $\alpha \varepsilon =-1$. In this conditions, if the Levi Civita connection of a manifold $(M,J,g)$ having an $(\alpha,\varepsilon)$-structure satisfies the next condition 
\[
(\nabla^{\mathrm{g}}_{X}J)Y-(\nabla^{\mathrm{g}}_{Y}J)X= 0,\enspace \forall X,Y\in \mathfrak{X}(M),
\] 
then it also satisfies  Codazzi-coupled conditions \eqref{eq:codazzi-coupled}, i.e., the Levi Civita connection is a Codazzi-coupled torsion-free connection.  Thus, one can recover our Theorem \ref{thm:kahler-codazzi} from the above one for the case $\alpha \varepsilon =-1$.

\begin{corollary}  Let $(M,J,g)$ be a $(J^2=\pm1)$-metric manifold such that $\alpha \varepsilon =-1$. If the next condition is fulfilled,
\begin{equation*}
(\nabla^{\mathrm{g}}_{X}J)Y-(\nabla^{\mathrm{g}}_{Y}J)X= 0,\enspace \forall X,Y\in \mathfrak{X}(M),
\end{equation*}
then $(M,J,g)$  is a K\"{a}hler type manifold.
\end{corollary}

Therefore, K\"{a}hler type manifolds satisfying $\alpha\varepsilon=-1$ are Codazzi-K\"{a}hler o Codazzi-para-K\"{a}hler manifolds such that the Levi Civita connection is a Codazzi-coupled connection.

\medskip

Table \ref{table:conditions} summarizes the results obtained about $(J^{2}=\pm 1)$-metric manifolds satisfying 
\[
(\nabla^{\mathrm{g}}_{X}J)Y + \alpha\varepsilon (\nabla^{\mathrm{g}}_{Y}J)X= 0,\enspace \forall X,Y\in \mathfrak{X}(M),
\]
according to the value of the product $\alpha\varepsilon=\pm1$ through the present paper.

\begin{table}[htb]
\begin{center}
\begin{tabular}{|c|c|c|}
 \hline
 \hline
Condition & $\alpha \varepsilon=-1$ & $\alpha \varepsilon=1$ \\
  \hline
  \hline
$(\nabla^{\mathrm{g}}_{X}J)Y+(\nabla^{\mathrm{g}}_{Y}J)X= 0 \Longleftrightarrow (\nabla^{\mathrm{g}}_X J) X=0$ & nearly K\"{a}hler type manifolds & K\"{a}hler type manifolds\\
 \hline
$(\nabla^{\mathrm{g}}_{X}J)Y-(\nabla^{\mathrm{g}}_{Y}J)X= 0$ & K\"{a}hler type manifolds & K\"{a}hler type manifolds\\
  \hline
  \hline
\end{tabular}
\end{center}
\caption{$(J^2=\pm1)$-metric manifolds satisfying $(\nabla^{\mathrm{g}}_{X}J)Y\pm(\nabla ^{g}_{Y}J)X= 0$}
\label{table:conditions}
\end{table}

\end{document}